\documentclass[3p,times]{elsarticle}

\usepackage{ecrc}
\pdfoutput=1
\volume{00}

\firstpage{1}

\journalname{Mathematics and Computers in Simulation}

\runauth{}

\jid{procs}

\jnltitlelogo{Procedia Computer Science}

\CopyrightLine{2011}{Published by Elsevier Ltd.}

\usepackage{amssymb}
\usepackage{amsthm}
\newcommand{\E}{\mathbb{E}}

\newcommand{\bdtau}{\boldsymbol{\tau}}
\newcommand{\bdv}{\boldsymbol{v}}
\newcommand{\x}{\boldsymbol{x}}

\newcommand{\R}{\mathbb{R}}

\newcommand{\abs}[1]{\left| #1 \right|}
\usepackage{leftindex}

\newcommand{\bfu}{{\mathbf{u}}}

\newcommand{\bfx}{{\mathbf{x}}}

\newcommand{\Indc}[1]{{\mathbf{1}_{\left\{{#1}\right\}}}}

\newtheorem{theorem}{Theorem}

\newtheorem{corollary}[theorem]{Corollary}
\newtheorem{proposition}[theorem]{Proposition}

\newtheorem{assumption}[theorem]{Assumption}
\newtheorem{remark}[theorem]{Remark}

\usepackage{algorithm}
\usepackage{algorithmicx}
\usepackage{algpseudocode}
\usepackage[figuresright]{rotating}
\usepackage{url}
\begin{document}

\begin{frontmatter}



\dochead{}
%% e.g. \dochead{17th International Conference on Dynamical Processes in Excited States of Solids}

\title{Flow Matching Transport for Quasi-Monte Carlo Integration}

\author[cjl]{Zhijun Zeng}
\author[cjl]{Jianlong Chen\corref{cor1}}
\address[cjl]{Department of Mathematical Sciences, Tsinghua University, Beijing 100084, China}
\cortext[cor1]{Corresponding author.}
\ead{E-mail address: zengzj22@mails.tsinghua.edu.cn (Z. Zeng).chen-jl22@mails.tsinghua.edu.cn (J. Chen).}
\begin{abstract}
High-dimensional integration with respect to complex target measures remains a fundamental challenge in computational science. While Flow Matching (FM) offers a powerful paradigm for constructing continuous-time transport maps, its deployment in high-precision integration is severely limited by the discretization bias inherent to numerical ODE solvers and the lack of rigorous convergence guarantees when coupled with Quasi-Monte Carlo (QMC) methods. This paper addresses these critical gaps by proposing Flow Matching Importance Sampling Quasi-Monte Carlo (FM-ISQMC), a framework designed to transform biased generative flows into unbiased, high-order integration schemes. Methodologically, we construct a transport map by composing a logistic base transformation with an Euler-discretized neural ODE field and employ importance sampling to correct for residual transport errors. Our central contribution is twofold. First, we establish a general convergence analysis for QMC importance sampling with arbitrary transport maps, identifying sufficient growth conditions for the $\mathcal{O}(N^{-1+\varepsilon})$ root-mean-square error rate. Second, we rigorously prove that the specific transport architecture of Flow Matching satisfies these conditions. Consequently, we establish a  $\mathcal{O}(N^{-1+\varepsilon})$ root-mean-square error for the unbiased FM-ISQMC estimator, extending classical QMC theory to the realm of generative models. Numerical experiments validate that FM-ISQMC consistently breaks through the error floor observed in direct transport methods, delivering superior precision. This work thus bridges the divide between deep generative modeling and numerical integration.
\end{abstract}


\end{frontmatter}

% \linenumbers

\section{Introduction}

High-dimensional numerical integration is a fundamental challenge in many fields of science and engineering, including Bayesian statistics, computational physics, and financial mathematics. The goal is often to compute the expectation of a function $f$ with respect to a complex, high-dimensional probability distribution $\pi$, that is, $\E_{X\sim\pi}[f(X)]$. The standard Monte Carlo (MC) method is robust and largely dimension-independent, but its convergence rate is limited to $O(N^{-1/2})$, where $N$ denotes the number of samples. To mitigate this fundamental limitation, Quasi-Monte Carlo (QMC) and randomized QMC (RQMC) methods have been introduced. Under sufficient regularity, these variants can achieve substantially faster convergence, often approaching $O(N^{-1})$ \cite{Niederreiter1992, owenscrambled}.

However, the efficiency of QMC methods is strongly influenced by the smoothness of the integrand and the geometry of the domain. When the target distribution $\pi$ is highly non-uniform or multimodal, applying QMC directly (e.g., via a simple inverse transform of a uniform measure) can be ineffective or theoretically invalid if the resulting integrand becomes singular or highly oscillatory. A powerful strategy to address this is \textit{measure transport} or \textit{transport maps} \cite{mosk:cafl:1996}. The idea is to construct a diffeomorphism $\bdtau: (0,1)^d \to \mathbb{R}^d$ that pushes forward a simple reference measure (e.g., the uniform distribution $\mathcal{U}(0,1)^d$) to the complex target $\pi$. If such a map is available, the integral can be transformed into an expectation over the hypercube, where QMC points (such as Sobol' sequences or lattice rules) can be effectively utilized.

In recent years, deep generative models have become a leading paradigm for learning transport maps from data or unnormalized densities. Normalizing flows (NFs) \cite{song2020score, liu2022flow} parameterize the map as a composition of invertible layers and have shown success in importance sampling and variational inference. More recently, Continuous Normalizing Flows (CNFs) \cite{liu2022flow} modeled by Ordinary Differential Equations (ODEs) have gained attention. Among these, \textit{Flow Matching} (FM) \cite{liu2022flow} has established itself as a state-of-the-art framework. Unlike traditional diffusion models or maximum likelihood training of CNFs, Flow Matching directly regresses a vector field that generates a probability path between the source and target distributions. This simulation-free training objective leads to more stable training and straighter transport paths, making FM appealing for the design of efficient samplers.

Despite the empirical success of flow-based samplers, the theoretical analysis of their integration with QMC has only recently begun to attract attention. For instance, Andral \cite{andral2024combining} empirically demonstrated that combining Normalizing Flows with RQMC yields variance reduction. On the theoretical front, Liu \cite{liu2024transport} proposed a Transport QMC framework with triangular maps, while Du and He \cite{du2025lp} established $L_p$-error rates for RQMC-SNIS with unbounded integrands. However, a specific analysis linking the architectural properties of Flow Matching (continuous-time flows) to QMC boundary growth conditions remains limited. The central issue is that QMC convergence theory requires strict regularity of the integrand. Specifically, for the transformed integrand $f(\bdtau(\mathbf{u}))$ to benefit from the $O(N^{-1+\epsilon})$ convergence rate of RQMC, it need to satisfy \textit{Owen's boundary growth condition}  \cite{Owen2006}. This condition restricts how fast the mixed partial derivatives of the function can grow as the input $\mathbf{u}$ approaches the boundary of the unit hypercube. While generative models optimize for global distribution matching (e.g., minimizing Kullback-Leibler divergence), they do not explicitly control the pointwise derivative growth at the boundaries. Furthermore, recent theoretical results suggest that the optimal velocity fields in FM can exhibit singularities at the boundaries \cite{gao2024convergence}, potentially violating the smoothness assumptions required for QMC.

In this paper, we propose a rigorous framework for Flow Matching Quasi-Monte Carlo (FM-QMC) and provide a comprehensive theoretical analysis of its convergence properties. We focus on the structural properties of transport maps generated by Flow Matching with a specific architectural choice: a logistic base map composed with an ODE flow driven by a bounded vector field.

Our main contributions are summarized as follows. First, we establish a general convergence analysis for the Importance Sampling QMC (ISQMC) estimator using arbitrary transport maps. We formulate a set of sufficient conditions for a general diffeomorphism to preserve the optimal QMC convergence rate, distinguishing between \textit{Value Growth Conditions} (controlling how fast points are mapped to the tails of the target distribution) and \textit{Derivative Growth Conditions} (controlling the explosion of the Jacobian and higher-order derivatives). Second, we prove that the transport maps constructed via Flow Matching satisfy these rigorous growth conditions. Specifically, we show that if the learned vector field and its derivatives are uniformly bounded (a condition satisfied by standard neural networks with bounded activation functions), and the ODE is solved via a fixed-step Euler method, then the resulting map $\bdtau$ behaves sufficiently well near the boundaries. To our knowledge, this is one of the first results explicitly linking the architectural properties of neural ODEs to the boundary growth conditions required for QMC theory. Third, based on the growth analysis, we establish that the FM-QMC estimator achieves a Root Mean Square Error (RMSE) of $O(N^{-1+\epsilon})$ for any $\epsilon > 0$, provided the target function $f$ has polynomial growth. This provides a solid theoretical foundation for using Flow Matching in high-precision integration tasks. Finally, we validate our theoretical findings through numerical experiments on various high-dimensional benchmarks, observing that FM-QMC consistently outperforms standard MC methods, exhibiting convergence rates consistent with our theoretical predictions.

The remainder of this paper is organized as follows. Section 2 reviews the background on Randomized Quasi-Monte Carlo integration and the Flow Matching framework. Section 3 presents the general theory of boundary growth conditions for transport maps. Section 4 constructs the FM transport map and provides the theoretical proofs verifying that it satisfies the required growth conditions (Theorems \ref{thm:value_growth} and \ref{thm:derivative_growth}). Section 5 presents numerical experiments demonstrating the efficiency and convergence rates of our method. Finally, Section 6 concludes the paper and discusses future directions.

\section{Background}\label{sec:background}
We introduce the background on quasi-Monte Carlo (QMC), randomized quasi-Monte Carlo (RQMC) and Flow Matching

\subsection{Quasi‐Monte Carlo and Randomized Quasi‐Monte Carlo}

Quasi‐Monte Carlo (QMC) integration estimates an integral of the form
\[
  \mu = \int_{[0,1]^d} f(\mathbf{x}) \,\mathrm{d}\mathbf{x}
\]
by replacing independent random sampling with a deterministic, low‐discrepancy point set \(\{\mathbf{u}_i\}_{i=1}^n\subset[0,1]^d\).  The QMC estimator is
\[
  \widehat\mu_n \;=\;\frac{1}{n}\sum_{i=1}^n f(\mathbf{u}_i)\,.
\]
The celebrated Koksma–Hlawka inequality \cite{koks:1942,hlaw:1961} shows that
\begin{equation}\label{equ: kh inequality}
      \bigl|\widehat\mu_n - \mu\bigr|
  \;\le\;
  V_{\mathrm{HK}}(f)\;\;D^{*}\bigl(\mathbf{u}_1,\dots,\mathbf{u}_n\bigr),
\end{equation}
where \(V_{\mathrm{HK}}(f)\) is the Hardy–Krause variation of \(f\), and the star discrepancy is
\[
  D^{*}\bigl(\mathbf{u}_1,\dots,\mathbf{u}_n\bigr)
  \;=\;
  \sup_{\mathbf{a}\in[0,1]^d}\,
  \Biggl|\frac{1}{n}\sum_{i=1}^n\mathbf{1}\{\mathbf{u}_i\in[0,\mathbf{a})\}
  \;-\;\prod_{j=1}^d a_j\Biggr|.
\]
A point set with a star discrepancy of $\mathcal{O}(n^{-1}(\log n)^{d})$ is referred to as a low-discrepancy point set \cite{nied:1992}. For the low-discrepancy point set, one can obtain an integration error of order \(O(n^{-1+\varepsilon})\) for every \(\varepsilon>0\), when \(V_{\mathrm{HK}}(f)<\infty\).

In practice, \(V_{\mathrm{HK}}(f)\) is often infinite, especially when \(f\) is unbounded or has singularities.  Randomized QMC (RQMC) applies a stochastic perturbation to a low‐discrepancy point set \(\{\mathbf{u}_i\}_{i=1}^n\subset[0,1]^d\) so that the randomized points \(\{\mathbf{u}_i'\}\) are each marginally \(\mathrm{Uniform}([0,1]^d)\) while retaining low discrepancy almost surely.  The RQMC estimator
\[
  \hat\mu'_n \;=\;\frac{1}{n}\sum_{i=1}^n f(\mathbf{u}_i')
\]
is unbiased (i.e. \(\mathbb{E}[\hat\mu'_n]=\mu\)), and its variance can be estimated from \(R\) independent replicates \(\{\hat\mu_n^{(r)}\}_{r=1}^R\) by
\[
  \widehat{\mathrm{Var}}\bigl(\hat\mu'_n\bigr)
  = \frac{1}{R-1}\sum_{r=1}^R\bigl(\hat\mu_n^{(r)} - \bar\mu_n\bigr)^2,
  \quad
  \bar\mu_n = \frac{1}{R}\sum_{r=1}^R \hat\mu_n^{(r)}.
\]
Common randomization schemes include:
\begin{enumerate}
\item \emph{Random shift modulo 1.}  Draw \(\boldsymbol{\Delta}\sim\mathrm{Uniform}([0,1]^d)\) and set
  \[
    \mathbf{u}_i' = (\mathbf{u}_i + \boldsymbol{\Delta}) \bmod 1.
  \]
  Under suitable smoothness and periodicity of \(f\),
  \(\mathrm{Var}[\hat\mu'_n] = O\bigl(n^{-2+\varepsilon}\bigr)\)
  for any \(\varepsilon>0\).
\item \emph{Digital shift.}  For a base-\(b\) digital net, draw
  \(\boldsymbol{\Delta}\sim\mathrm{Uniform}(\{0,\dots,b-1\}^d)\) and define
  \[
    \mathbf{u}_i' = \mathbf{u}_i \oplus \boldsymbol{\Delta},
  \]
  where \(\oplus\) denotes digit-wise addition modulo \(b\).  This yields
  \(\mathrm{Var}[\hat\mu'_n] = o\bigl(n^{-1}\bigr)\)
  whenever \(\mathbb{E}[f(\mathbf{u})^2]<\infty\).
\item \emph{Owen’s scrambling.}  A nested uniform scramble of a \((t,m,s)\)-net produces
  \[
    \mathrm{Var}[\hat\mu'_n] = O\bigl(n^{-3+\varepsilon}\bigr)
    \quad\text{for any }\varepsilon>0,
  \]
  provided \(f\) has square‐integrable mixed first derivatives \cite{smoovar,localanti}.
\end{enumerate}
We refer to \cite{dick:pill:2010,dick:kuo:sloa:2013} for a comprehensive introduction of QMC and \cite{lecu:lemi:2002} for a review of RQMC.

\subsection{Growth condition}\label{subsec:GrowthCondition}
We clarify the notation used for partial derivatives. Let $1\!:\!d = \{1, \dots, d\}$. For a set  $v \subseteq \{1, \dots, d\}$, the symbol $\partial^v$ denotes the mixed partial derivative with respect to the variables indexed by $v$, defined as $\partial^v = \prod_{j \in v} \frac{\partial}{\partial u_j}$. For a multi-index $\lambda = (\lambda_1, \dots, \lambda_d) \in \mathbb{N}_0^d$, $\partial^\lambda$ denotes the derivative operator $\frac{\partial^{|\lambda|}}{\partial x_1^{\lambda_1} \dots \partial x_d^{\lambda_d}}$, where $|\lambda| = \sum_{i=1}^d \lambda_i$. Then we review the boundary growth condition for an integrand $h:[0,1]^d\to\R$, as proposed by \cite{Owen2006}.

\begin{assumption}[Boundary growth condition]\label{assump: growth}
For arbitrarily small $B>0$, there exists a constant $C>0$ such that
\begin{align}
    \label{equ: growth}
    |\partial^v h(\bfu)|\leq C \prod_{j=1}^d [\min(u_j,1-u_j)]^{-B-\Indc{j\in v}}
\end{align}
for any $v\subseteq1{:}d$ and $\bfu\in(0,1)^d$.
\end{assumption}
Assumption~\ref{assump: growth} limits the growth rate of integrand and its mixed first-order partial derivatives near the boundaries. As a coordinate $u_j$ approaches 0 or 1, the derivative $\partial^v h(\bfu)$ is bounded by the inverse distance to the boundary raised to a certain power. 

The following theorem is adapted from Theorem 5.7 of \cite{Owen2006}.
\begin{theorem}[Adapted from Theorem 5.7 of \cite{Owen2006}]\label{thm:owengrowth}
    If the integrand $h$ satisfies the boundary growth condition~\ref{assump: growth}, then the scrambled net estimator $\hat\mu_n$ achieves a root mean square error (RMSE) of order $O(n^{-1+\epsilon})$ for arbitrarily small $\epsilon>0$.
\end{theorem}

The analysis of the growth of derivatives for composite functions relies on the multivariate Fa\`a di Bruno formula.
\begin{theorem}[Multivariate Fa\`a di Bruno Formula \cite{cons:savi:1996}]
\label{thm:faa_di_bruno}
Let $f: \mathbb{R}^d \to \mathbb{R}$ and $\bdtau: \mathbb{R}^d \to \mathbb{R}^d$ be smooth functions. For $h = f \circ \bdtau$ and any non-empty set $v \subseteq \{1, \dots, d\}$, we have
$$
\partial^v h(\bfu) = \sum_{\lambda \in \mathbb{N}_0^d : 1 \le |\lambda| \le |v|} (\partial^\lambda f)(\bdtau(\bfu)) \sum_{s=1}^{|v|} \sum_{(\mathbf{k}, \boldsymbol{\ell}) \in \mathcal{A}(\lambda, s)} \prod_{r=1}^s \partial^{\ell_r} \bdtau_{k_r}(\bfu),
$$
where
$$
\mathcal{A}(\lambda, s) = \left\{ (\mathbf{k}, \boldsymbol{\ell}) = (k_1, \dots, k_s, \ell_1, \dots, \ell_s) : 
 k_r \in \{1, \dots, d\}, \ell_r \subseteq \{1, \dots, d\}, 
 |\{j \in \{1, \dots, s\} : k_j = i\}| = \lambda_i, \bigsqcup_{r=1}^s \ell_r = v 
\right\}.
$$
Here, the notation $\bigsqcup_{r=1}^s \ell_r = v$ indicates that the sets $\ell_1, \dots, \ell_s$ form a partition of $v$, that is, the sets $\ell_1, \dots, \ell_s$ are pairwise disjoint and their union is $v$.

\end{theorem}
\subsection{Transport map and importance sampling}\label{subsec: Transport map and importance sampling}
QMC and its randomized variant (RQMC) produce point sets that uniformly cover the unit cube $[0,1]^d$. In many applications, we must sample from an intractable target distribution $\pi$ on $\mathbb{R}^d$ to estimate $E_{X\sim\pi}[f(X)]$.  
\iffalse
When $p$ admits a tractable inverse cumulative distribution function, exact transformation of uniform samples is possible.  For instance, to draw from the standard Gaussian law $\mathcal{N}(0,I_d)$, one may apply the component‐wise inverse CDF $\Phi^{-1}$ to each QMC or RQMC point.  Equivalently, the Box–Muller transform
\(
  Z_1 \;=\;\sqrt{-2\ln u_1}\,\sin(2\pi u_2), 
  \quad
  Z_2 \;=\;\sqrt{-2\ln u_1}\,\cos(2\pi u_2),
\)
maps two independent draws $u_1,u_2\sim\mathrm{Uniform}(0,1)$ into independent samples $Z_1,Z_2\sim\mathcal{N}(0,1)$.
\fi
Thus we want to find a measurable map
\[
  \bdtau\colon [0,1]^d \;\longrightarrow\;\mathbb{R}^d
  \quad\text{such that}\quad
  \mathbf{x}=\bdtau(\mathbf{u})\sim \pi
  \quad\text{for}\quad
  \mathbf{u}\sim\mathrm{Uniform}\bigl([0,1]^d\bigr).
\]
In this setting, $\bdtau$ \emph{pushes forward} the uniform distribution on $[0,1]^d$ to the target distribution $\pi$, and is referred to as a \emph{transport map}. We denote the pushforward of a density $q$ by a map $\bdtau$ as $\bdtau_{\#} q$. Although closed‐form transport maps are available only in special cases, one may approximate $\bdtau$ within a flexible function class so that the induced density
\[
  q_{\bdtau}(\mathbf{x})
  = \bigl(\bdtau_{\#}\,q_u\bigr)(\mathbf{x})
  \quad\text{with}\quad
  q_u(\mathbf{u})=\mathbf{1}_{[0,1]^d}(\mathbf{u})
\]
closely matches $p(\mathbf{x})$.  Such approximate transport maps enable the direct application of QMC and RQMC sampling to a broad array of target distributions. If $\bdtau: (0,1)^d \to \R^d$ is a $C^1$ diffeomorphism, then the density $q_{\bdtau}$ is given by the change-of-variable formula:
\begin{equation}
q_{{\bdtau}}(\mathbf{x}) = |\text{det} J_{{\bdtau}} ({\bdtau}^{-1}(x))|^{-1} \label{equ: change of variable}
\end{equation}
where $J_{{\bdtau}}$ is the Jacobian matrix of the transformation ${\bdtau}$ , defined as
\[ J_{{\bdtau}}(\mathbf{u})= \left(\frac{\partial {\bdtau}_i}{\partial u_j} \right)_{1\leq i,j\leq d}.\]

Therefore, to estimate the expectation $\mu = \E_{X\sim\pi}[f(X)]$, we consider to find a $C^1$ diffeomorphism $\bdtau: (0,1)^d \to \R^d$ that transforms the uniform distribution $\mathcal{U}(0,1)^d$ to a proposal distribution $q_{\bdtau}$. If we have a good transport map ${\bdtau}$ such that $q_{\bdtau} \approx \pi$, the expectation $\mathbb{E}_\pi[f(X)]$ is then approximated by $\mathbb{E}_{\mathbf{U} \sim \mathcal{U}(0,1)^d}[f \circ {\bdtau}(\mathbf{U})]$, which can be estimated by
\begin{align}
\frac{1}{n} \sum_{i=1}^n (f \circ \bdtau)(\bfu_i)
\end{align} 
with the point set $\{\bfu_i\}_{i=1}^n \subset (0,1)^d$.

However, if \(q_{\bdtau} \neq \pi\), the estimator incurs a bias. We can correct this bias by importance sampling. Note that
\begin{equation}
E_{X\sim \pi}[f(X)] = E_{Y \sim q_{{\bdtau}}}[\frac{f\pi}{q_{{\bdtau}}}(Y)]=E_{\mathbf{U}\sim \mathcal{U}(0,1)^d}[\frac{f\pi}{q_{{\bdtau}}}\circ {\bdtau}(\mathbf{U})]. \notag
\end{equation}
Then using \eqref{equ: change of variable} , the integrand of interest writes
$$
h(\bfu) \coloneqq \left(\frac{f\pi}{q_{\bdtau}}\circ\bdtau\right)(\bfu) = (f \circ \bdtau)(\bfu) \cdot (\pi \circ \bdtau)(\bfu) \cdot \abs{\det J_{\bdtau}(\bfu)}.
$$
Thus the expectation $\mu$ can be estimated by
\begin{align}
    \hat{I}_n:= \frac{1}{n} \sum_{i=1}^n h(\bfu_i) = \frac{1}{n} \sum_{i=1}^n (f \circ \bdtau)(\bfu_i) \cdot (\pi \circ \bdtau)(\bfu_i) \cdot \abs{\det J_{\bdtau}(\bfu_i)}
\end{align}
with the point set $\{\bfu_i\}_{i=1}^n \subset (0,1)^d$. If we use points $\bfu_i \sim \mathrm{Uniform}([0,1]^d)$, such as identically distributed (i.i.d.) samples or RQMC points, then $\hat{I}_n$ is an \textbf{unbiased} estimator of $\mu$.

In Section~\ref{sec:convergence}, we will prove that $h(\bfu)$ satisfies  boundary growth condition under some assumptions for $f,\pi$ and $\bdtau$, thereby ensuring the $O(n^{-1+\epsilon})$ RMSE convergence rate for estimator $\hat{I}_n$ with scrambled nets $\{\bfu_i\}_{i=1}^n$.

\subsection{Flow Matching and Conditional Flow Matching}
Let $p_0$ denote a simple prior distribution (e.g., Gaussian or Logistic) and $p_1 = \pi$ a complex target distribution on $\mathbb{R}^d$. The objective of flow-based generative modeling is to construct a continuous mapping, or \emph{flow}, that transports $p_0$ to $p_1$. To construct such a flow, we consider a family of conditional probability paths. Given a data sample $z$, a conditional probability path $\rho(x,t \mid z)$ for $t\in[0,1]$ is a probability flow satisfying
\begin{equation}
    \begin{cases}
        \rho(x,0 \mid z) = p_0(x), \\
        \rho(x,1 \mid z) = \delta(x-z).
    \end{cases}
\end{equation}
Such a conditional probability path can be induced by a flow on $\mathbb{R}^d$. Let $v:[0,1]\times\mathbb{R}^d\to\mathbb{R}^d$ be a time-dependent vector field that generates a diffeomorphic map $\phi:[0,1]\times\mathbb{R}^d\to\mathbb{R}^d$ through the ordinary differential equation
\begin{equation}
    \begin{cases}
        \dfrac{d\phi_t(x|z)}{dt} = v(\phi_t(x|z),t \mid z), \\
        \phi_0(x|z) = x.
    \end{cases}
\end{equation}

A particularly simple and feasible choice for the velocity field is the straight-line flow 
\begin{equation}
    v(x,t \mid z) = \frac{z - x}{1 - t}.
\end{equation}
The corresponding conditional probability path then satisfies the continuity equation
\begin{equation}
    \frac{\partial \rho(x,t \mid z)}{\partial t}
    + \nabla_x \cdot \big(\rho(x,t \mid z)\, v(x,t \mid z)\big)
    = 0,
\end{equation}
in the sense of distributions.

To obtain a flow that transports $p_0$ to $p_1$, we multiply both sides by $p_1(z)$ and integrate over $z$, yielding
\begin{equation}
    \int_{\mathbb{R}^d} p_1(z)
    \left(
        \frac{\partial \rho(x,t \mid z)}{\partial t}
        + \nabla_x \cdot \big(\rho(x,t \mid z)v(x,t \mid z)\big)
    \right) dz = 0, 
    \qquad \forall x\in\mathbb{R}^d.
\end{equation}
Define
\[
    \hat{\rho}(x,t) := \int_{\mathbb{R}^d} p_1(z)\,\rho(x,t\mid z)\,dz,
\]
which satisfies
\begin{equation}
    \begin{aligned}
        \hat{\rho}(x,0) &= \int_{\mathbb{R}^d} p_1(z)\,p_0(x)\,dz = p_0(x), \\
        \hat{\rho}(x,1) &= \int_{\mathbb{R}^d} p_1(z)\,\delta(x-z)\,dz = p_1(x).
    \end{aligned}
\end{equation}
Thus, $\hat{\rho}(x,t)$ defines a probability path connecting $p_0$ and $p_1$.

The resulting flow satisfies the continuity equation
\begin{equation}
    \frac{\partial \hat{\rho}(x,t)}{\partial t} 
    + \nabla_x \cdot \big(\hat{\rho}(x,t)\,v(x,t)\big)
    = 0,
\end{equation}
where the average velocity field $u$ is given by
\begin{equation}\label{eq:average_velocity}
    v(x,t)
    :=
    \frac{
        \int_{\mathbb{R}^d} p_1(z)\,\rho(x,t\mid z)\, v(x,t\mid z)\,dz
    }{
        \hat{\rho}(x,t)
    }.
\end{equation}

The goal of \emph{Flow Matching} (FM) is to directly learn a neural approximation $v(x,t;\theta)$ of the velocity field $v(x,t)$ by minimizing
\begin{equation}
    \mathcal{L}_{\mathrm{FM}}(\theta)
    := \mathbb{E}_{t \sim \mathcal{U}[0,1],\,x_t \sim \hat{\rho}(\cdot,t)}
       \big\| v(x_t,t;\theta) - u(x_t,t) \big\|_2^2 
    = \int_0^1 
      \int_{\mathbb{R}^d}
      \left\| 
        v(x,t;\theta)
        - 
        \frac{
            \int_{\mathbb{R}^d}p_1(z)\rho(x,t\mid z)v(x,t\mid z)\,dz
        }{
            \hat{\rho}(x,t)
        }
      \right\|_2^2
      \hat{\rho}(x,t)\,dx\,dt.
    \label{eq:fm_loss}
\end{equation}
However, this formulation is often impractical since computing the averaged velocity field $v(x,t)$ is typically computationally expensive.

To address this issue, one may exchange the order of expectation and conditioning, leading to the \emph{conditional flow matching} (CFM) objective
\begin{equation}
    \mathcal{L}_{\mathrm{CFM}}(\theta)
    := \mathbb{E}_{t \sim \mathcal{U}[0,1],\, z \sim p_1,\, x_t \sim \rho(\cdot,t\mid z)}
       \big\| v(x_t,t;\theta) - v(x_t,t \mid z) \big\|_2^2 
    = \int_0^1\!\!
      \int_{\mathbb{R}^d}\!\!
      \int_{\mathbb{R}^d}
      \|v(x,t;\theta) - v(x,t\mid z)\|_2^2
      \rho(x,t\mid z)\,p_1(z)\,dz\,dx\,dt.
    \label{eq:cfm_loss}
\end{equation}
It can be shown that the FM objective \eqref{eq:fm_loss} and the CFM objective \eqref{eq:cfm_loss} yield identical gradients\cite{lipmanflow}
\[
\nabla_\theta \mathcal{L}_{\mathrm{FM}}(\theta)
=
\nabla_\theta \mathcal{L}_{\mathrm{CFM}}(\theta),
\]
while the latter is significantly more tractable in practice.

When the prior $p_0$ is Gaussian, another common and analytically convenient choice for the conditional path $\rho(x,t\mid z)$ is the Gaussian family
\begin{equation}
    \rho(x,t\mid z) 
    = \mathcal{N}\!\big(x\,\big|\,\mu_t(z),\, \sigma_t^2(z) I\big),
    \label{eq:gaussian_cond_path}
\end{equation}
where $\mu_t(z)$ and $\sigma_t(z)$ are the time-dependent mean and standard deviation, respectively. The corresponding conditional flow map and velocity field are given by
\begin{equation}
\left\{
\begin{array}{ccl}
       \phi_t(x\mid z) & = & \sigma_t(z)\, x + \mu_t(z),\\[0.3em]
       v(x,t\mid z) & = & \displaystyle
       \frac{\dot{\sigma}_t(z)}{\sigma_t(z)}\,
       \big(x - \mu_t(z)\big)
       + \dot{\mu}_t(z),
\end{array}
\right.
    \label{eq:gaussian_map_velocity}
\end{equation}
where the overdot denotes differentiation with respect to $t$.

A particularly important instance of the Gaussian construction arises when the mean and variance interpolate linearly between the endpoints. Let
\begin{equation}
    \mu_t(z) = t\,z, 
    \qquad 
    \sigma_t(z) = 1 - (1 - \sigma)\,t,
    \label{eq:ot_linear_interp}
\end{equation}
for some prescribed terminal standard deviation $\sigma\in[0,1]$. In this case, particles move along straight trajectories in $\mathbb{R}^d$, and the corresponding flow $\phi_t(\cdot\mid z)$ coincides with the Wasserstein--2 optimal transport displacement interpolation between the initial Gaussian $p_0=\mathcal{N}(0,I)$ and the terminal Gaussian $p_1(\cdot\mid z)=\mathcal{N}(z,\sigma^2 I)$, thereby defining a geodesic in the Wasserstein metric. However, at $t=1$, the resulting marginal density becomes
\(
    \hat{\rho}(x,1) 
    = p_1 \ast \mathcal{N}(\boldsymbol{0},\sigma^2 I),
\)
that is, a Gaussian-smoothed version of $p_1$. Taking the limit $\sigma\to 0$ recovers the \textit{straight-line flow} described above. For further variants and design choices for the velocity field, we refer to \cite{tongimproving, kornilov2024optimal, chen2018neural} .
\section{Transport maps for QMC and RQMC with flow matching}\label{sec:transport_map_flow_matching}
In this section we explain how flow--based generative models give rise to transport maps that are well suited for QMC and RQMC integration.

The probability--flow ODE formulation of flow matching describe a deterministic flow that transports a simple prior distribution \(p_0\) to a complex target distribution \(p_1\) on \(\mathbb{R}^d\).  
Let \(\bdv : \mathbb{R}^d \times [0,1] \to \mathbb{R}^d\) denote the (average) velocity field as in the flow matching formulation (see \eqref{eq:average_velocity}).  
The associated flow \(\Psi : \mathbb{R}^d\times[0,1]\to\mathbb{R}^d\) is defined as the solution of the initial–value problem
\begin{equation}\label{forward}
  \partial_t \Psi(\mathbf{x},t) = \bdv(\Psi(\mathbf{x},t),t),
  \qquad
  \Psi(\mathbf{x},0) = \mathbf{x},
  \quad \mathbf{x} \sim p_0.
\end{equation}
At the terminal time, the induced transport map
\[
  \tilde{\bdtau}(\mathbf{x}) := \Psi(\mathbf{x},1)
\]
pushes forward \(p_0\) to \(p_1\), that is,
\(
  (\tilde{\bdtau})_{\#} p_0 \approx p_1.
\)

To interface this construction with QMC and RQMC on the unit cube, we combine the learned flow with a base transformation \(G:[0,1]^d\to\mathbb{R}^d\) that maps the uniform distribution on \([0,1]^d\) to the prior \(p_0\). The resulting composite map
\[
  \bdtau := \tilde{\bdtau}\circ G : [0,1]^d \longrightarrow \mathbb{R}^d
\]
is a transport map of the form considered in the previous subsection: it pushes the uniform distribution on \([0,1]^d\) onto the data distribution \(p_1\).

We also recall how the flow can be coupled with the evolution of the log–density along the trajectory.  In the continuous normalizing flow (CNF) literature (see, e.g., \cite{chen2018neural,behrmann2019invertible,ben2022matching}), the continuity equation associated with the velocity field \(\bdv\) implies the instantaneous change–of–variables identity
\[
  \frac{d}{dt}\log \hat{\rho}\bigl(\Psi(\mathbf{x},t),t\bigr)
  = - \nabla_\mathbf{x} \cdot \bdv\bigl(\Psi(\mathbf{x},t),t\bigr),
\]
where \(\hat{\rho}(\mathbf{x},t)\) denotes the density of \(\Psi(\cdot,t)_{\#}p_0\). Let $l(t) = \log \hat{\rho}\bigl(\Psi(\mathbf{x},t),t\bigr)$, the state \(\Psi(\mathbf{x},t)\) and its log–density can be advanced jointly by the augmented ODE
\begin{equation}\label{eq:joint_ode_state_logp}
  \frac{d}{dt}
  \begin{bmatrix}
    \Psi(\mathbf{x},t)\\[0.3em]
    \ell(t)
  \end{bmatrix}
  =
  \begin{bmatrix}
    \bdv\bigl(\Psi(\mathbf{x},t),t\bigr)\\[0.3em]
    -\,\nabla_\mathbf{x} \cdot \bdv\bigl(\Psi(\mathbf{x},t),t\bigr)
  \end{bmatrix},
  \qquad \begin{bmatrix}
       \ell(0)\\ \Psi(\mathbf{x},0)\end{bmatrix} =\begin{bmatrix}
       \log p_0(\mathbf{x})\\ \mathbf{x}\end{bmatrix} 
\end{equation}
so that \(\ell(1) = \log p_1(\Psi(\mathbf{x},1))\).

While the continuous formulation provides a theoretical basis, practical implementation requires two key approximations: discretizing the time integral and parameterizing the velocity field.
In practice, the true velocity field \(\bdv\) is unknown and is approximated by a neural network \(\bdv_\theta(\mathbf{x}, t)\) with parameters \(\theta\). Thus the ODE in \eqref{forward} is replaced by
\begin{equation}\label{forward_approx}
  \partial_t \Psi_\theta(\mathbf{x},t) = \bdv_\theta(\Psi_\theta(\mathbf{x},t),t),
  \qquad
  \Psi_\theta(\mathbf{x},0) = \mathbf{x}.
\end{equation}

To compute the map \(\tilde{\bdtau}(\mathbf{x})\), we we must solve the ODE numerically. We consider the Forward Euler discretization with $N$ steps of size $h = 1/N$. Let $t_k = kh$. The step map ${\bdtau}^k: \mathbb{R}^d \to \mathbb{R}^d$ is defined as 
\[
{\bdtau}^k(\mathbf{x}) = \mathbf{x} + h \bdv_\theta(\mathbf{x}, t_k).
\]
The total transport map $\tilde{\bdtau}$ is then approximated by the composition of these step maps as
\[
{\bdtau}^{N-1} \circ \dots \circ {\bdtau}^0(\mathbf{x}).
\]

To interface this construction with QMC and RQMC on the unit cube, we combine the learned flow with a base transformation \(G:[0,1]^d\to\mathbb{R}^d\) that maps the uniform distribution on \([0,1]^d\) to the prior \(p_0\). A choice is the component-wise inverse logistic function (logit function), defined as 
\[
G(u) = \log\left(\frac{u}{1-u}\right).
\]
Since $G$ is differentiable and monotonically increasing, its inverse corresponds to the CDF of the logistic distribution.

The final composite transport map is thus
\[
  \bdtau(\mathbf{u}) := {\bdtau}^{N-1} \circ \dots \circ {\bdtau}^0 \circ G(\mathbf{u}).
\]
This map $\bdtau: (0,1)^d \to \mathbb{R}^d$ pushes the uniform distribution on \([0,1]^d\) onto the data distribution \(p_1\).

Once the transport map \(\bdtau\) is constructed via flow matching, we can use the content introduced in Section~\ref{sec:background} to estimate expectations using (R)QMC points. We call the resulting method Flow Maching Importance Sampling QMC (FM-ISQMC).
Below we will conduct a detailed convergence order analysis in Section~\ref{sec:convergence}
 and Section~\ref{sec:flow_matching_growth}. Specifically, Section~4 proposes assumptions for general \(\bdtau\) from \((0,1)^d\) to \(\mathbb{R}^d\) and performs the analysis; Section~\ref{sec:flow_matching_growth} verifies that the \(\bdtau\) constructed using flow matching satisfies the assumptions of Section~\ref{sec:convergence}
, so that the conclusions of Section~\ref{sec:convergence} hold for such \(\bdtau\).

\section{Convergence analysis for general transport maps with QMC}
\label{sec:convergence}

In this section, we analyze the convergence of RQMC estimator for estimating expectation $\mu = \E_{X\sim\pi}[f(X)]$ using a general transport map $\bdtau: (0,1)^d \to \R^d$ with importance sampling. Recall subsection~\ref{subsec: Transport map and importance sampling} where we defined the integrand of interest as
$$h(\bfu) = (f \circ \bdtau)(\bfu) \cdot (\pi \circ \bdtau)(\bfu) \cdot \abs{\det J_{\bdtau}(\bfu)}.$$
We will show that under certain assumptions for $f$, $\pi$, and $\bdtau$, the integrand $h(\bfu)$ satisfies boundary growth condition (Assumption \ref{assump: growth}). This will ensure that the RQMC estimator $\hat{I}_n = \frac{1}{n} \sum_{i=1}^n h(\bfu_i)$ achieves an RMSE convergence rate of $O(n^{-1+\epsilon})$ for any $\epsilon > 0$ when using scrambled nets.
\subsection{Assumptions}

Our analysis rests on the following set of assumptions regarding the growth and decay of the integrand, the target density and the transport map's derivatives. 

\begin{assumption}[Integrand $f$ Growth]
\label{assum:f}
Assume there exists a univariate CDF $F(x)$ such that for all multi-indices $\lambda \in \mathbb{N}_0^d$ with $|\lambda| \le d$, and for arbitrarily small $B > 0$, there exists $C_1 > 0$ such that for all $\bfx \in \R^d$, the following bound holds:
$$
|\partial^\lambda f(\bfx)| \le C_1 \prod_{k=1}^d \left[ \min(F(x_k), 1-F(x_k)) \right]^{-B}
$$
\end{assumption}

\begin{assumption}[Target $\pi$ Decay]
\label{assum:pi}
Assume that $\pi(\bfx)$ and its derivatives decay exponentially. Specifically, there exist constants $C_2 > 0$ and $\alpha > 0$ such that for all multi-indices $\lambda \in \mathbb{N}_0^d$ with $|\lambda| \le d$, we have
$$
|\partial^\lambda \pi(\bfx)| \le C_2 \prod_{k=1}^d e^{-\alpha |x_k|}
$$
\end{assumption}

\begin{assumption}[Transport growth]
\label{assum:tau} 
Let $F(x)$ be the same CDF as in Assumption~\ref{assum:f}. Assume that for any $v \subseteq \{1, \dots, d\}$, $m \in \{1, \dots, d\}$, and $j \in \{1, \dots, d\}$, and for arbitrarily small $B > 0$, there exists $C_3 > 0$ such that for all $\bfu \in (0,1)^d$,
\begin{align}
|\partial^v(\partial^{\{m\}}\bdtau_j)(\bfu)| \le C_3 \cdot \min(u_m, 1-u_m)^{-1} \cdot \prod_{k=1}^d \min(u_k, 1-u_k)^{-B - \Indc{k \in v}} \label{term:tau_deriv}
\end{align}
Moreover, there exists constants $B_0>0, C_4>0$ such that for any $k \in \{1, \dots, d\}$,
\begin{align} \label{eq:tau_inv_growth}
\min(F (\tau(\mathbf{u})_k), 1 - F (\tau(\mathbf{u})_k))^{-1} \le C_4 \prod_{l=1}^{d} \min(u_l, 1 - u_l)^{-B_0}.
\end{align}
Additionally, we assume the transport map grows sufficiently fast near the boundary. Specifically, there exists a constant $C_5 > 1/\alpha$ and a constant $C'\in \mathbb{R}$ such that
\begin{align} \label{eq:tau_lower_bound}
|\tau(\mathbf{u})_j| \ge C_5 |\ln(\min(u_j, 1-u_j))| - C'.
\end{align}
\end{assumption}

\begin{remark}
The inequality~\eqref{term:tau_deriv} in Assumption \ref{assum:tau} implies that
$$
|\partial^v(\bdtau_j)(\bfu)| \le C_3 \cdot \prod_{k=1}^d \min(u_k, 1-u_k)^{-B - \Indc{k \in v}}.
$$
\end{remark}

\subsection{Growth Analysis of Integrand Components}

We decompose the integrand $h(\bfu) = (f \circ \bdtau)(\bfu) \cdot (\pi \circ \bdtau)(\bfu) \cdot \abs{\det J_{\bdtau}(\bfu)}$ and bound the growth of each term's partial derivatives.

\begin{proposition}[Growth of $f \circ \bdtau$]
\label{prop:h_f}
Under Assumptions \ref{assum:f} and \ref{assum:tau}, for arbitrarily small $B > 0$, there exists a constant $C > 0$ such that $f \circ \bdtau$ satisfies the inequality
$$
|\partial^v(f \circ \bdtau(\bfu))| \le C \prod_{k=1}^d \left[ \min(u_k, 1-u_k) \right]^{-B - \Indc{k \in v}}.
$$
\end{proposition}
\begin{proof}
We apply the multivariate Faa di Bruno formula to $\partial^v (f \circ \bdtau)$. A generic term in the expansion is given by
$$
T = (\partial^\lambda f)(\bdtau(\bfu)) \cdot \prod_{l=1}^{|\lambda|} \partial^{v_l} \bdtau_{j_l}(\bfu)
$$
where $\lambda$ is a multi-index with $|\lambda|$ equal to the number of blocks in the partition, and $\{v_l\}$ partitions $v$.
By Assumption \ref{assum:f}, evaluating at $\bfx = \bdtau(\bfu)$, we obtain
\begin{align}
|(\partial^\lambda f)(\bdtau(\bfu))| \le C_1 \prod_{k=1}^d \left[ \min(F(\tau(\bfu)_k), 1-F(\tau(\bfu)_k)) \right]^{-B}. \label{eq:f_bound}
\end{align}
Using the value growth condition from Assumption \ref{assum:tau}, namely
$$
\left[ \min(F(\tau(\bfu)_k), 1-F(\tau(\bfu)_k)) \right]^{-1} \le C_4 \prod_{l=1}^{d} \min(u_l, 1 - u_l)^{-B_0},
$$
substituting this bound into \eqref{eq:f_bound}, we obtain there exists a constant $\tilde{C} > 0$ such that
$$
|(\partial^\lambda f)(\bdtau(\bfu))| \le \tilde{C} \prod_{l=1}^d \min(u_l, 1 - u_l)^{-d B_0 B}.
$$
For the product of derivatives, Assumption \ref{assum:tau} implies
$$
\left| \prod_{l=1}^{|\lambda|} \partial^{v_l} \bdtau_{j_l} \right| \le C_3 \prod_{k=1}^d \min(u_k, 1-u_k)^{-1 - B - \Indc{k \in v}}.
$$
Multiplying these bounds yields
$$
|T| \le C \prod_{k=1}^d \min(u_k, 1-u_k)^{-B_0 - 1 - B - \Indc{k \in v}}.
$$
Since $B_0$ and $B$ are arbitrarily small, their combination $B_0 + 1 + B$ can be written as $1 + B'$ where $B'$ is arbitrarily small. Relabeling $B'$ as $B$, the exponent simplifies to $-B - \Indc{k \in v}$, which matches the required form.
\end{proof}

\begin{proposition}[Growth of $\pi \circ \bdtau$ ]
\label{prop:h_pi}
Under Assumptions \ref{assum:pi} and \ref{assum:tau}, for arbitrarily small $B > 0$, there exists a constant $C > 0$ such that $\pi \circ \bdtau$ satisfies the bound
$$
|\partial^v(\pi \circ \bdtau(\bfu))| \le C \prod_{k=1}^d \left[ \min(u_k, 1-u_k) \right]^{1 - B - \Indc{k \in v}}.
$$
\end{proposition}
\begin{proof}
We apply the multivariate Faa di Bruno formula to $\partial^v (\pi \circ \bdtau)$. A generic term is given by
$$
T = (\partial^\lambda \pi)(\bdtau(\bfu)) \cdot \prod_{l=1}^{|\lambda|} \partial^{v_l} \bdtau_{j_l}(\bfu)
$$
where $\lambda$ is a multi-index with $|\lambda|$ equal to the number of blocks in the partition, and $\{v_l\}$ partitions $v$.
By Assumption \ref{assum:pi}, the first factor is bounded by
$$
|(\partial^\lambda \pi)(\bdtau(\bfu))| \le C_2 \prod_{k=1}^d e^{-\alpha |\tau(\bfu)_k|}.
$$
Using the growth condition \eqref{eq:tau_lower_bound} from Assumption \ref{assum:tau}, we have $|\tau(\bfu)_k| \ge C_5 |\ln(\min(u_k, 1-u_k))| - C'$. Thus, we have
$$
e^{-\alpha |\tau(\bfu)_k|} \le e^{-\alpha (C_5 |\ln(\min(u_k, 1-u_k))| - C')} = e^{\alpha C'} \min(u_k, 1-u_k)^{\alpha C_5}.
$$
Substituting this back, we obtain
$$
|(\partial^\lambda \pi)(\bdtau(\bfu))| \le \tilde{C} \prod_{k=1}^d \min(u_k, 1-u_k)^{\alpha C_5}.
$$
For the product of derivatives, Assumption \ref{assum:tau} provides the bound
$$
\left| \prod_{l=1}^{|\lambda|} \partial^{v_l} \bdtau_{j_l} \right| \le C_3 \prod_{k=1}^d \min(u_k, 1-u_k)^{ - B - \Indc{k \in v_l}}.
$$
Multiplying these bounds, we have the bound for the generic term $T$ as
$$
|T| \le C \prod_{k=1}^d \min(u_k, 1-u_k)^{\alpha C_5 - B - \Indc{k \in v}} \le C \prod_{k=1}^d \min(u_k, 1-u_k)^{\alpha 1 - B - \Indc{k \in v}},
$$
where we use $\alpha C_5 > 1$ and $\min(u_k, 1-u_k) \le 1$ in the last inequality.
\end{proof}

\begin{proposition}[Growth of $\abs{\det J_{\bdtau}}$]
\label{prop:h_det}
Under Assumption \ref{assum:tau}, for arbitrarily small $B > 0$, there exists a constant $C > 0$ such that the determinant of a general dense Jacobian satisfies the following growth bound
$$
|\partial^v(\abs{\det J_{\bdtau}(\bfu)})| \le C \prod_{k=1}^d \left[ \min(u_k, 1-u_k) \right]^{-B - 1 - \Indc{k \in v}}.
$$
\end{proposition}

\begin{proof}
Since $\bdtau$ is a diffeomorphism, $\det J_{\bdtau}(\bfu)$ is differentiable and non-zero for all $\bfu \in (0,1)^d$. By the continuity of the determinant, $\det J_{\bdtau}(\bfu)$ maintains a constant sign. Therefore, we have $|\partial^v(\abs{\det J_{\bdtau}(\bfu)})| = |\partial^v(\det J_{\bdtau}(\bfu))|$.

The determinant of the general Jacobian $J_{\bdtau}(\bfu)$ is given by the Leibniz formula
$$
\det J_{\bdtau}(\bfu) = \sum_{\sigma \in S_d} \text{sgn}(\sigma) \underbrace{\prod_{i=1}^d \partial^{\{\sigma(i)\}}\bdtau_i(\bfu)}_{P_\sigma(\bfu)}
$$
where $S_d$ denotes the set of all permutations of $\{1, \dots, d\}$.
By the triangle inequality, $\abs{\partial^v(\det J_{\bdtau}(\bfu))} \le \sum_{\sigma \in S_d} \abs{\partial^v(P_\sigma(\bfu))}$. We only need to bound the growth of a single term $\partial^v(P_\sigma)$.

We apply the general Leibniz rule to $\partial^v(P_\sigma)$. A generic term $T$ in this expansion has the form
$$
T = \prod_{i=1}^d \partial^{v_i}\left(\partial^{\{\sigma(i)\}}\bdtau_i(\bfu)\right)
$$
where $\{v_1, \dots, v_d\}$ is a partition of $v$, meaning $\bigsqcup_{i=1}^d v_i = v$.

We now apply Assumption \ref{assum:tau} to each of the $d$ factors in $T$. For the $i$-th factor, we set $j=i$, $m=\sigma(i)$, and $v'=v_i$. This gives the bound
$$
\abs{\partial^{v_i}(\partial^{\{\sigma(i)\}}\bdtau_i)} \le C_3 \cdot \min(u_{\sigma(i)}, 1-u_{\sigma(i)})^{-1} \cdot \prod_{k=1}^d \min(u_k, 1-u_k)^{-B - \Indc{k \in v_i}}
$$
To bound $|T|$, we multiply these $d$ bounds together to obtain
\begin{align}
|T| 
& \le \prod_{i=1}^d \left[ C_3 \cdot \min(u_{\sigma(i)}, 1-u_{\sigma(i)})^{-1} \cdot \left( \prod_{k=1}^d \min(u_k, 1-u_k)^{-B - \Indc{k \in v_i}} \right) \right] \notag \\
&\le C_3^d \prod_{k=1}^d \min(u_k, 1-u_k)^{-1 - dB - \sum_{i=1}^d \Indc{k \in v_i}} \notag \\
&= C_3^d \prod_{k=1}^d \min(u_k, 1-u_k)^{-1 - B' - \Indc{k \in v}} \notag \\
&\le C \prod_{k=1}^d \min(u_k, 1-u_k)^{-B' - 1 - \Indc{k \in v}} \notag
\end{align}
where in the second inequality we combined the contributions to the exponent of each $\min(u_k, 1-u_k)$, in the third equality we used that $\{v_1, \dots, v_d\}$ is a partition of $v$, and in the last inequality we set $C = C_3^d$ and $B' = dB$.

%Let's analyze the exponent for a generic term $\min(u_k, 1-u_k)$. From the $m=\sigma(i)$ terms, we have $\prod_{i=1}^d \min(u_{\sigma(i)}, 1-u_{\sigma(i)})^{-1} = \prod_{k=1}^d \min(u_k, 1-u_k)^{-1}$, since $\sigma$ is a permutation of $\{1, \dots, d\}$. This contributes $-1$ to the exponent of each $\min(u_k, 1-u_k)$. From the $v'=v_i$ terms, we have $\prod_{i=1}^d \prod_{k=1}^d \min(u_k, 1-u_k)^{-B - \Indc{k \in v_i}}$. This contributes $\sum_{i=1}^d (-B - \Indc{k \in v_i})$ to the exponent. Combining these, the total exponent for $\min(u_k, 1-u_k)$ is $\left( \sum_{i=1}^d (-B - \Indc{k \in v_i}) \right) - 1 = -dB - \left( \sum_{i=1}^d \Indc{k \in v_i} \right) - 1$.
%We now analyze this exponent based on whether $k \in v$. If $k \notin v$ (un-differentiated variable), then $k \notin v_i$ for all $i$, so $\sum_{i=1}^d \Indc{k \in v_i} = 0$. The exponent is $-dB - 0 - 1 = -dB - 1$. If $k \in v$ (differentiated variable), then $k$ must be in at least one $v_i$, so $\sum_{i=1}^d \Indc{k \in v_i} \ge 1$. The exponent is less than or equal to $-dB - 1 - 1 = -dB - 2$. Letting $B_3 = dB$ (which is arbitrarily small), we can write a single bound covering both cases: the exponent is less than or equal to $-B_3 - 1 - \Indc{k \in v}$. Thus, $|T| \le C_3^d \prod_{k=1}^d \min(u_k, 1-u_k)^{-B_3 - 1 - \Indc{k \in v}}$. 
Since $B$ is arbitrarily small, $B' = dB$ is also arbitrarily small. Thus we obtain the required bound.
\end{proof}

\subsection{Main Convergence Theorem}

We now combine these bounds to prove the main result.

\begin{theorem}[Integrand $h(\bfu)$ satisfies boundary growth condition]
\label{thm:main_convergence}
Under Assumptions \ref{assum:f}, \ref{assum:pi}, and \ref{assum:tau}, the integrand $h(\bfu) = (f \circ \bdtau)(\bfu) \cdot (\pi \circ \bdtau)(\bfu) \cdot \abs{\det J_{\bdtau}(\bfu)}$ satisfies boundary growth condition, stated in Assumption \ref{assump: growth}.
\end{theorem}
\begin{proof}
We apply the general Leibniz rule to $h(\bfu)$. A generic term $T$ in the expansion of $\partial^v h(\bfu)$ is given by
$$
T = \partial^{v_1}((f \circ \bdtau)(\bfu)) \cdot \partial^{v_2}((\pi \circ \bdtau)(\bfu)) \cdot \partial^{v_3}(\abs{\det J_{\bdtau}(\bfu)})
$$
where $\{v_1, v_2, v_3\}$ is a partition of $v$. We bound $|T|$ by multiplying the bounds from Propositions \ref{prop:h_f}, \ref{prop:h_pi}, and \ref{prop:h_det}:
\begin{align*}
    |T| &\le \left( C \prod_{k=1}^d \min(u_k, 1-u_k)^{-B - \Indc{k \in v_1}} \right) \\
    &\quad \times \left( C \prod_{k=1}^d \min(u_k, 1-u_k)^{1 - B - \Indc{k \in v_2}} \right) \\
    &\quad \times \left( C \prod_{k=1}^d \min(u_k, 1-u_k)^{-1 - B - \Indc{k \in v_3}} \right) \\
    &= C^3 \prod_{k=1}^d \min(u_k, 1-u_k)^{-3B - (\Indc{k \in v_1} + \Indc{k \in v_2} + \Indc{k \in v_3})} \\
    &= \tilde{C} \prod_{k=1}^d \min(u_k, 1-u_k)^{-B' - \Indc{k \in v}},
\end{align*}
where $\tilde{C}$ is a constant and $B' = 3B$. The last equality holds because $\{v_1, v_2, v_3\}$ is a partition of $v$, so for any $k$, $\Indc{k \in v_1} + \Indc{k \in v_2} + \Indc{k \in v_3} = \Indc{k \in v}$.
This final bound matches the form required by boundary growth condition (Assumption \ref{assump: growth}) with an arbitrarily small constant $B'$. Since every term in the Leibniz expansion of $\partial^v h(\bfu)$ satisfies this bound, the function $h(\bfu)$ itself satisfies the condition.
\end{proof}

\begin{corollary}
By Theorem \ref{thm:main_convergence} and the properties of scrambled nets, the RQMC estimator $\hat{\mu}_n = \frac{1}{n} \sum_{i=1}^n h(\bfu_i)$ for the expectation $\mu = \E_{X\sim\pi}[f(X)]$ achieves an RMSE convergence rate of $O(n^{-1+\epsilon})$ for any $\epsilon > 0$.
\end{corollary}

\section{Growth Condition Verification for Flow Matching Transport Map}
\label{sec:flow_matching_growth}

In this section, we provide a detailed theoretical justification for the transport map constructed via Flow Matching. We prove it satisfies Assumption \ref{assum:tau} in Section~\ref{sec:convergence}. Thus, when combined with integrands satisfying Assumption~\ref{assum:f} and target density $\pi$ satisfying Assumption~\ref{assum:pi}, the resulting composite function meets boundary growth condition, ensuring the desired convergence rates for our FM-ISQMC estimator.

\subsection{Setup and Definitions}

We first recall the setting in Section~\ref{sec:transport_map_flow_matching}. Consider the transport map $\bdtau: (0,1)^d \to \mathbb{R}^d$ constructed as the composition of a base transformation $G$ and a discretized flow map.

\paragraph{Base Transformation}
The base transformation $G$ is a bijection from $(0, 1)^d$ to $\mathbb{R}^d$. It applies the same univariate function $G : (0, 1) \to \mathbb{R}$ to all the $d$ components. For simplicity of notation, we define
$$
G(\mathbf{u}) = (G(u_1), \dots, G(u_d))^\top.
$$
In this work, we choose $G$ to be the inverse logistic function, also known as the logit function, defined as $G(u) = \log\left(\frac{u}{1-u}\right)$. Since $G$ is differentiable and monotonically increasing, its inverse function is well defined and corresponds to the CDF of a distribution supported on $\mathbb{R}$. For this reason, we write $G = F^{-1}$, where $F$ is the CDF of the logistic distribution on $\mathbb{R}$, given by $F(x) = (1+e^{-x})^{-1}$. The derivatives of $G$ satisfy the following growth condition
\begin{equation} \label{eq:g_growth}
    |G^{(k)}(u)| \le C_k \cdot [\min(u, 1-u)]^{-k}, \quad \forall k \ge 1.
\end{equation}

\paragraph{Flow Matching Discretization}
The flow matching ODE is given by $\dot{\mathbf{x}} = \bdv_\theta(\mathbf{x}, t)$. We consider the Forward Euler discretization with $N$ steps of size $h = 1/N$. Let $t_k = kh$. The step map ${\bdtau}^k: \mathbb{R}^d \to \mathbb{R}^d$ is defined as ${\bdtau}^k(\mathbf{x}) = \mathbf{x} + h \bdv_\theta(\mathbf{x}, t_k)$. The total transport map $\bdtau$ is the composition $\bdtau(\mathbf{u}) = {\bdtau}^{N-1} \circ \dots \circ {\bdtau}^0 \circ G(\mathbf{u})$. We denote the composition of the base map and the first $k$ steps as ${\bdtau}^{0:k} = {\bdtau}^{k-1} \circ \dots \circ {\bdtau}^0 \circ G$, with ${\bdtau}^{0:0} = G$.

\paragraph{Assumptions on the Vector Field}
We assume the learned vector field $\bdv_\theta(\mathbf{x}, t)$ is smooth and has bounded derivatives of all orders required for the analysis.
\begin{assumption} \label{assum:v_bounded}
    For any multi-index $\alpha \in \mathbb{N}_0^{d+1}$, including the case $\alpha=0$, there exists a constant $M_\alpha$ such that
    \begin{equation}
        \sup_{\mathbf{x} \in \mathbb{R}^d, t \in [0,1]} |\partial^\alpha \bdv_\theta(\mathbf{x}, t)| \le M_\alpha.
    \end{equation}
\end{assumption}
This assumption requires that both the vector field $\bdv_\theta$ itself and its derivatives are uniformly bounded.

\begin{remark}
    Standard Multi-Layer Perceptrons (MLPs) with common activation functions can satisfy Assumption \ref{assum:v_bounded}. Consider an MLP $f(\mathbf{x}) = W_L \sigma(W_{L-1} \dots \sigma(W_1 \mathbf{x}) \dots)$. If the activation function $\sigma$ is bounded and has bounded derivatives of all orders, such as Sigmoid or Tanh, then the entire network and its derivatives are bounded, satisfying the assumption. If $\sigma$ is unbounded but has bounded derivatives, such as Softplus, SiLU, or GELU, the network output may grow linearly. However, for the derivatives with $\alpha \ge 1$, the chain rule involves $\sigma'$, which is bounded. Thus, the derivatives of the network are bounded. Note that while Assumption \ref{assum:v_bounded} strictly requires $\bdv_\theta$ to be bounded for $\alpha=0$, for the growth condition proof below, we primarily rely on the boundedness of the derivatives for $\alpha \ge 1$. The boundedness of $\bdv_\theta$ itself is useful to ensure the trajectory does not escape to infinity in finite time, but the critical part for the singularity analysis is the control of the derivatives.
\end{remark}

\iffalse
\begin{remark}[On Time Singularities]
    We acknowledge that Assumption \ref{assum:v_bounded} is a strong condition. Recent theoretical work \cite{gao2024convergence} indicates that for certain target distributions, the optimal velocity field in flow matching may exhibit singularities as $t \to 1$. In such cases, the uniform boundedness assumption would be violated. However, in practice, we parameterize $\bdv_\theta$ using neural networks with bounded weights and activations (or regularized unbounded activations), which enforces boundedness on the learned field, potentially at the cost of approximation error near the singularity. Our analysis pertains to the properties of the \textit{learned} map under these architectural constraints.
\end{remark}
\fi

\subsection{Proof of Growth Condition}

We first verify the second part of Assumption \ref{assum:tau} regarding value growth by establishing the following theorem.

\begin{theorem}[Value Growth of Flow Matching Map] \label{thm:value_growth}
    Let $\bdtau(\mathbf{u}) = \tilde{{\bdtau}}(G(\mathbf{u}))$ be the transport map, where $G$ is the component-wise inverse logistic map and $\tilde{{\bdtau}}$ is the flow map generated by a vector field $\bdv_\theta$ satisfying Assumption \ref{assum:v_bounded}. Then, the following properties hold.
    \begin{enumerate}
        \item There exists a constant $C_4 > 0$ such that for all $j \in \{1, \dots, d\}$ and $\mathbf{u} \in (0,1)^d$,
        \begin{equation}
            \min((F \circ \bdtau)(\mathbf{u})_j, 1 - (F \circ \bdtau)(\mathbf{u})_j)^{-1} \le C_4 \min(u_j, 1 - u_j)^{-1}.
        \end{equation}
        This implies that the inequality \eqref{eq:tau_inv_growth} of Assumption \ref{assum:tau} holds with parameters $B_0 = 1$.
        \item The map satisfies the logarithmic lower bound growth condition. Specifically, there exists a constant $C'$ such that
        \begin{equation}
            |{\bdtau}(\mathbf{u})_j| \ge |\ln(\min(u_j, 1-u_j))| - C'.
        \end{equation}
        This implies that the inequality \eqref{eq:tau_lower_bound} of Assumption \ref{assum:tau} holds with parameters $C_5 = 1$.
    \end{enumerate}
\end{theorem}

\begin{proof}
    Since $\bdv_\theta$ is bounded, there exists $M > 0$ such that $|\tilde{{\bdtau}}(\mathbf{x}) - \mathbf{x}|_\infty \le M$. Let $x_j = G(\mathbf{u})_j$. Then $x_j - M \le {\bdtau}_j(\mathbf{u}) \le x_j + M$.
    
    We first establish the inverse boundary control. Using the monotonicity of $F$ and the identity $e^{-x_j} = (1-u_j)/u_j$, we have
    $$
    F({\bdtau}_j) \ge F(x_j - M) = \frac{1}{1 + e^{-x_j}e^M} = \frac{u_j}{u_j + (1-u_j)e^M} \ge e^{-M} u_j.
    $$
    Similarly, for the upper tail, we obtain
    $$
    1 - F({\bdtau}_j) \ge 1 - F(x_j + M) = \frac{e^{-x_j}e^{-M}}{1 + e^{-x_j}e^{-M}} = \frac{(1-u_j)e^{-M}}{u_j + (1-u_j)e^{-M}} \ge e^{-M}(1-u_j).
    $$
    Combining these yields $\min(F({\bdtau}_j), 1-F({\bdtau}_j)) \ge e^{-M} \min(u_j, 1-u_j)$, which implies the result with $C_4=e^M$ and $B=1$.

    Next, we derive the logarithmic lower bound growth. We have $|{\bdtau}_j(\mathbf{u})| \ge |x_j| - M = |G(\mathbf{u})_j| - M$.
    Recall $G(\mathbf{u})_j = \log\left(\frac{u_j}{1-u_j}\right)$.
    We observe that $|\log\left(\frac{u}{1-u}\right)| \ge |\ln(\min(u, 1-u))| - \ln 2$.
    To see this, consider $u \le 1/2$. Then $\min(u, 1-u) = u$.
    $\log\left(\frac{u}{1-u}\right) = \ln u - \ln(1-u)$. Since $1/2 \le 1-u < 1$, we have $-\ln 2 \le \ln(1-u) < 0$.
    Thus $\log\left(\frac{u}{1-u}\right) \in [\ln u, \ln u + \ln 2]$. Since $\ln u$ is negative and large, $|\log\left(\frac{u}{1-u}\right)| \ge |\ln u| - \ln 2$.
    The case $u > 1/2$ is symmetric.
    Therefore,
    $$
    |{\bdtau}_j(\mathbf{u})| \ge |\ln(\min(u_j, 1-u_j))| - (M + \ln 2).
    $$
    This directly verifies the condition with $C_5 = 1$ and $C' = M + \ln 2$.
\end{proof}

Next, we establish the derivative growth condition for $\bdtau$.

\begin{theorem}[Derivative Growth of Flow Matching Map] \label{thm:derivative_growth}
    Let $\bdtau$ be the transport map constructed as the composition of the base map $G$ and $N$ Euler steps with a vector field $\bdv_\theta$ satisfying Assumption \ref{assum:v_bounded}. Then, for any $m \in \{1, \dots, d\}$ and any set of indices $v \subseteq \{1, \dots, d\}$, and for arbitrarily small $B > 0$, there exists a constant $C_3 > 0$ such that the partial derivatives of $\bdtau$ satisfy the inequality
    \begin{equation} \label{eq:target_condition}
        \abs{\partial^{v}(\partial^{\{m\}}{\bdtau}_j)(\mathbf{u})} \le C_3 \cdot \min(u_m, 1-u_m)^{-1} \cdot \prod_{k=1}^d \min(u_k, 1-u_k)^{-B - \Indc{k \in v}}.
    \end{equation}
\end{theorem}

\begin{proof}
We prove that for any $0 \le k \le N$, the map ${\bdtau}^{0:k}$ satisfies the growth condition \eqref{eq:target_condition}. The proof proceeds by induction on $k$.

We begin with the base case $k=0$. Let ${\bdtau}^{0:0} = G$. We check if ${\bdtau}^{0:0}$ satisfies the condition. Since $G$ is diagonal, $({\bdtau}^{0:0})_j(\mathbf{u}) = G(u_j)$. The derivative $\partial^{\{m\}} ({\bdtau}^{0:0})_j$ is non-zero only if $j=m$, in which case it equals $G'(u_m)$. For a general multi-index $v$, the term $\partial^v (\partial^{\{m\}} ({\bdtau}^{0:0})_m)(\mathbf{u})$ is non-zero only if $v \subseteq \{m\}$. If $v=\emptyset$, the term is $G'(u_m)$, bounded by $C_3 \min(u_m, 1-u_m)^{-1}$. If $v=\{m\}$, the term is $G''(u_m)$, bounded by $C_3 \min(u_m, 1-u_m)^{-2}$. In both cases, the bound $|\partial^v \partial^{\{m\}} ({\bdtau}^{0:0})_j| \le C_3 \min(u_m, 1-u_m)^{-1} \prod_{k \in v} \min(u_k, 1-u_k)^{-1}$ holds, satisfying the condition with $B=0$.

For the inductive step, assume that the map ${\bdtau}^{0:k}$ satisfies the growth condition \eqref{eq:target_condition}. Consider the next step map ${\bdtau}^k(\mathbf{x}) = \mathbf{x} + h \bdv_\theta(\mathbf{x}, t_k)$. We want to show that ${\bdtau}^{0:k+1} = {\bdtau}^k \circ {\bdtau}^{0:k}$ also satisfies the condition. The $j$-th component is $({\bdtau}^{0:k+1})_j(\mathbf{u}) = ({\bdtau}^{0:k})_j(\mathbf{u}) + h v_{\theta, j}({\bdtau}^{0:k}(\mathbf{u}), t_k)$. Differentiating with respect to $u_m$ and then by $\partial^v$, we get $\partial^v \partial^{\{m\}} ({\bdtau}^{0:k+1})_j = \partial^v \partial^{\{m\}} ({\bdtau}^{0:k})_j + h \partial^v \partial^{\{m\}} [v_{\theta, j}({\bdtau}^{0:k}(\mathbf{u}), t_k)]$. The first term $\partial^v \partial^{\{m\}} ({\bdtau}^{0:k})_j$ satisfies the bound by the induction hypothesis. We need to analyze the second term.

We analyze the term $\partial^v \partial^{\{m\}} (v_{\theta, j} \circ {\bdtau}^{0:k})$ by first applying the chain rule for $\partial^{\{m\}}$ and then the Leibniz rule for $\partial^v$.
$$
\partial^{\{m\}} (v_{\theta, j} \circ {\bdtau}^{0:k}) = \sum_{p=1}^d (\partial^{\{p\}} v_{\theta, j}) \circ {\bdtau}^{0:k} \cdot \partial^{\{m\}} ({\bdtau}^{0:k})_p.
$$
Applying $\partial^v$ to this product yields a sum over $p$ and over partitions $v_A \uplus v_B = v$:
$$
\partial^v \partial^{\{m\}} (v_{\theta, j} \circ {\bdtau}^{0:k}) = \sum_{p=1}^d \sum_{v_A \subseteq v} \partial^{v_A} [(\partial^{\{p\}} v_{\theta, j}) \circ {\bdtau}^{0:k}] \cdot \partial^{v_B} [\partial^{\{m\}} ({\bdtau}^{0:k})_p].
$$
For the first factor $\partial^{v_A} [(\partial^{\{p\}} v_{\theta, j}) \circ {\bdtau}^{0:k}]$, we apply the multivariate Fa\`a di Bruno formula. Since $\bdv_\theta$ and its derivatives are bounded, and using the induction hypothesis (treating any index in $v_A$ as the primary index $m$ leads to a bound $\prod_{k \in v_A} \min(u_k, 1-u_k)^{-1-B}$), this factor is bounded by $C \prod_{k \in v_A} \min(u_k, 1-u_k)^{-1-B}$.
For the second factor $\partial^{v_B} [\partial^{\{m\}} ({\bdtau}^{0:k})_p]$, the induction hypothesis directly gives the bound $C \min(u_m, 1-u_m)^{-1} \prod_{k \in v_B} \min(u_k, 1-u_k)^{-1-B}$.
Multiplying these bounds, and using $v_A \uplus v_B = v$, we obtain
$$
\left| \partial^v \partial^{\{m\}} (v_{\theta, j} \circ {\bdtau}^{0:k}) \right| \le \tilde{C} \min(u_m, 1-u_m)^{-1} \prod_{k \in v} \min(u_k, 1-u_k)^{-1-B}.
$$
This matches the required growth condition. Since $v_{\theta, j}$ and its derivatives are uniformly bounded, they do not alter the growth rate. Therefore, ${\bdtau}^{0:k+1}$ satisfies the same growth condition as ${\bdtau}^{0:k}$.

By induction, we conclude that ${\bdtau}^{0:k}$ satisfies the growth condition for all $0 \le k \le N$. Since $\bdtau = {\bdtau}^{0:N}$, the final transport map satisfies Assumption \ref{assum:tau}.
\end{proof}

\iffalse
\begin{remark}[Dependence on Constants]
    The constant $C_3$ in Theorem \ref{thm:derivative_growth} implicitly depends on the number of Euler steps $N$ and the dimension $d$. Specifically, the repeated application of the chain rule in the inductive step suggests that $C_3$ may grow exponentially with $N$ in the worst case. However, since each step is a perturbation of the identity scaled by the step size, the effective growth is controlled. A fully explicit tracking of these constants is beyond the scope of this work but is an important direction for future refinement.
\end{remark}
\fi

\subsection{Convergence Analysis of the FM-ISQMC Estimator}

We analyze the convergence of the numerical integration using the constructed transport map $\bdtau$. The goal is to estimate the expectation $\mu = \E_{X\sim\pi}[f(X)]$. We employ the FM-ISQMC estimator defined as
\begin{equation}
    \hat{I}_n = \frac{1}{n} \sum_{i=1}^n h(\mathbf{u}_i) = \frac{1}{n} \sum_{i=1}^n (f \circ \bdtau)(\mathbf{u}_i) \cdot (\pi \circ \bdtau)(\mathbf{u}_i) \cdot \abs{\det J_{\bdtau}(\mathbf{u}_i)},
\end{equation}
where $\{\mathbf{u}_i\}_{i=1}^n$ are scrambled net points.

\begin{theorem}[Convergence Rate of FM-ISQMC] \label{thm:convergence_rate}
    Suppose the vector field $\bdv_\theta$ satisfies Assumption \ref{assum:v_bounded}, the integrand $f$ satisfies Assumption \ref{assum:f} , and the target density $\pi$ satisfies Assumption \ref{assum:pi}. 
    Then, the integrand $h(\mathbf{u})$ satisfies  boundary growth condition (Assumption \ref{assump: growth}). Consequently, the FM-ISQMC estimator for the expectation $\mu = \E_{X\sim\pi}[f(X)]$ achieves an RMSE convergence rate of $O(n^{-1+\epsilon})$ for any $\epsilon > 0$.
\end{theorem}

\begin{proof}
    First, by Theorem~\ref{thm:value_growth} and Theorem~\ref{thm:derivative_growth}, the condition that $\bdv_\theta$ satisfies Assumption~\ref{assum:v_bounded} implies that the transport map $\bdtau$ satisfies the  Assumption~\ref{assum:tau}.
    Next, since $f$ satisfies Assumption~\ref{assum:f}, $\pi$ satisfies Assumption~\ref{assum:pi}, and $\bdtau$ satisfies Assumption \ref{assum:tau}, we can use Theorem~\ref{thm:main_convergence} from Section~\ref{sec:convergence}. This theorem guarantees that the importance sampling integrand $h(\mathbf{u})$ satisfies boundary growth condition (Assumption~\ref{assump: growth}).
    
    Finally, applying the Theorem~\ref{thm:owengrowth} for scrambled nets on functions satisfying boundary growth condition, we conclude that the RMSE of the estimator $\hat{I}_n$ is $O(n^{-1+\epsilon})$.
\end{proof}

\section{Experimental results}
\label{sec:numerics}

In our numerical experiments we investigate the performance of flow-based transport maps combined with randomized quasi-Monte Carlo (RQMC) sampling. Throughout, we use scrambled Sobol' point sets in $[0,1)^d$ as the underlying RQMC design; these points are mapped to a simple factorized logistic base distribution $p_0$ on $\R^d$ via its inverse CDF and then transported to the target distribution $\pi$ by the fitted flow. The velocity field $v_\theta(x,t)$ is parameterized by a residual fully connected network that combines a linear embedding of the state $x$, a Fourier feature embedding of time $t$, several residual MLP blocks with layer normalization and Mish activations, and a final linear projection to $\R^d$. Given the learned transport map $\tau$ (obtained by numerically integrating the probability-flow ODE from $t=0$ to $t=1$ starting from $p_0$), we estimate for each coordinate $j=1,\dots,d$ the first and second moment $\E_{X\sim\pi}[X_j]$ and $\E_{X\sim\pi}[X_j^2]$. Let $\{\xi_i^{(S)}\}_{i=1}^N$ denote base samples drawn either independently from $p_0$ ($S=\mathrm{MC}$) or from a scrambled Sobol' sequence mapped to $\R^d$ through the logistic inverse CDF ($S=\mathrm{QMC}$), and set $x_i^{(S)} = \tau(\xi_i^{(S)})$. The flow-matching Monte Carlo and RQMC estimators (FM-MC and FM-QMC) are then given by
\[
  \widehat{I}_N^{\mathrm{FM}\text{-}S}(f)
  \;=\;
  \frac{1}{N}\sum_{i=1}^N f\bigl(x_i^{(S)}\bigr),
  \qquad S\in\{\mathrm{MC},\mathrm{QMC}\},
\]
while the corresponding importance sampling variants (FM-ISMC and FM-ISQMC) correct from the proposal $q_\tau$ back to the target via the self-normalized estimator
\[
  \widehat{I}_N^{\mathrm{FM}\text{-}\mathrm{IS}S}(f)
  \;=\;
  \frac{\sum_{i=1}^N w_i^{(S)} f\bigl(x_i^{(S)}\bigr)}
       {\sum_{i=1}^N w_i^{(S)}},
  \qquad
  w_i^{(S)} \;=\; \frac{\pi\bigl(x_i^{(S)}\bigr)}{q_\tau\bigl(x_i^{(S)}\bigr)},
  \quad S\in\{\mathrm{MC},\mathrm{QMC}\}.
\]
We assess the accuracy of these four estimators for the coordinate-wise first and second moments on three representative posterior distributions: a two-dimensional Gaussian mixture, a two-dimensional banana-shaped distribution, and a 30-dimensional Gaussian mixture. Our code is available at \url{https://github.com/yl602019618/QMC_flow_matching}

\subsection{Mixture of Gaussians}
\subsubsection{2D case}
We first consider a two-dimensional mixture of Gaussian distributions as a simple but nontrivial multimodal target. The target density $\pi$ is a four-component Gaussian mixture with equal weights,
\[
  \pi(x)
  \;=\;
  \frac{1}{4}\sum_{k=1}^{4} \mathcal{N}\bigl(x \,\big|\, \mu_k,\Sigma_k\bigr),
  \qquad x\in\mathbb{R}^2,
\]
where the component means (written as column vectors) are
\[
  \mu_1 = (1,1)^\top,\quad
  \mu_2 = (2,3.6)^\top,\quad
  \mu_3 = (3.3,2.8)^\top,\quad
  \mu_4 = (1.1,2.9)^\top,
\]
and the corresponding covariance matrices are obtained by a uniform scaling of $1/4^2$,
\[
  \Sigma_1 = \frac{1}{40^2}
  \begin{pmatrix}
    2   & 0.6\\
    0.6 & 1
  \end{pmatrix},\quad
  \Sigma_2 = \frac{1}{4^2}
  \begin{pmatrix}
    2     & -0.4\\
    -0.4 & 2
  \end{pmatrix},\quad
  \Sigma_3 = \frac{1}{4^2}
  \begin{pmatrix}
    3   & 0.8\\
    0.8 & 2
  \end{pmatrix},\quad
  \Sigma_4 = \frac{1}{4^2}
  \begin{pmatrix}
    3 & 0\\
    0 & 0.5
  \end{pmatrix}.
\]

% We generate $10^4$ i.i.d.\ samples from this mixture and use them to train the flow-matching model described in Section~\ref{sec:numerics}, taking as base distribution an independent logistic law on $\mathbb{R}^2$ with zero location and unit scale. The model is trained for $5\times 10^4$ gradient steps using the Adam optimizer. During training and sampling, trajectories are integrated forward from $t=0$ to $t=1$ using the Heun scheme with $32$ time steps, while density evaluation relies on backward integration of the probability–flow ODE with a fourth-order Runge–Kutta method and exact divergence computation.

Figure~\ref{fig:1} summarizes the quality of the learned transport map. As shown in Figs.~\ref{fig:1}(a)–(b), the learned transport accurately recovers the multimodal structure of the target distribution. Figs.~\ref{fig:1}(c)–(d) compare the true mixture density with the density estimated by the flow-matching model. The estimated density closely follows the analytical Gaussian mixture over the region of interest, indicating that the model provides a faithful approximation of $\pi$ both in sample space and at the level of the underlying density.

We next examine the impact of flow-based transport and RQMC sampling on the accuracy of moment estimation. Using the four estimators defined in Section~\ref{sec:numerics} (FM-MC, FM-QMC, FM-ISMC, and FM-ISQMC), we estimate the first and second moment of the two-dimensional mixture and, for each method and each sample size $N$, compute the root-mean-square error (RMSE) over $10$ independent repetitions. The two Monte Carlo-based methods (FM-MC and FM-ISMC) exhibit the expected $O(N^{-1/2})$ convergence. The FM-QMC estimator initially displays a convergence rate of nearly $O(N^{-1})$, consistent with the enhanced variance reduction provided by RQMC, but its error curve eventually saturates as $N$ increases, reflecting the dominant influence of residual approximation error in the learned transport map. In contrast, FM-ISQMC achieves a convergence rate that is substantially higher than that of the Monte Carlo-based schemes and maintains a steady decrease in log-RMSE across the entire range of sample sizes considered, illustrating that combining flow-based importance sampling with RQMC can effectively exploit both variance reduction and bias correction in this setting.
\begin{figure}[!htbp]
  \centering
  \includegraphics[width=0.75\textwidth]{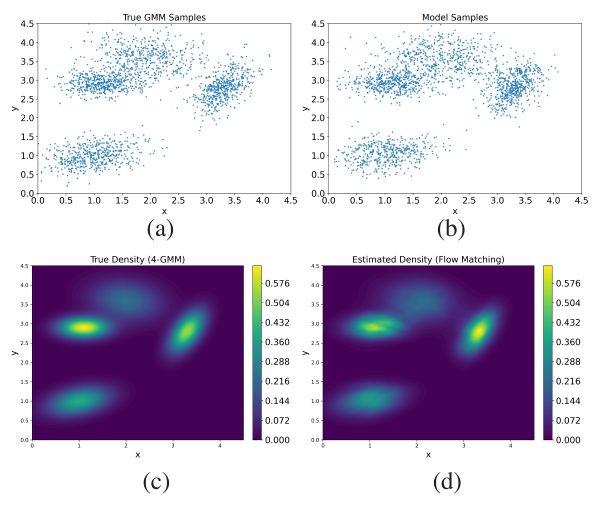}
  \caption{(a) Reference samples from the 2D mixture of Gaussian distribution. (b) Samples generated by the flow-matching model. (c) True density of the 2D GMM distribution. (d) Density estimated by the flow-matching model.}
  \label{fig:1}
\end{figure}

\begin{figure}[!htbp]
  \centering
  \includegraphics[width=0.75\textwidth]{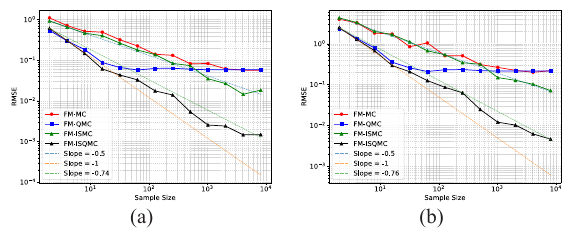}
  \caption{Log-RMSEs for estimating (a) the first and (b) the second moment of the 2D mixture of Gaussian distribution. Each RMSE is computed over 10 independent repetitions.}
  \label{fig:2}
\end{figure}

\subsubsection{30D case}
We next investigate a higher-dimensional example, namely a 30-dimensional mixture of Gaussian distributions with four equally weighted components. The target density $\pi$ is given by
\[
  \pi(x)
  \;=\;
  \frac{1}{4}\sum_{k=1}^{4} \mathcal{N}\bigl(x \,\big|\, m_k, \Sigma\bigr),
  \qquad x\in\mathbb{R}^{30},
\]
where the component means $m_k\in\mathbb{R}^{30}$ are defined by
\[
  m_1 = (-2,-2,0,\dots,0)^\top,\quad
  m_2 = ( 2,-2,0,\dots,0)^\top,\quad
  m_3 = (-2, 2,0,\dots,0)^\top,\quad
  m_4 = ( 2, 2,0,\dots,0)^\top,
\]
and all components share the same diagonal covariance matrix
\[
  \Sigma = 0.5\,I_{30}.
\]
Thus, the multimodality is concentrated in the first two coordinates, while the remaining $28$ coordinates are independent Gaussian directions with identical marginal variance. The flow-matching model is trained on i.i.d.\ samples from this mixture using the same architecture as in the two-dimensional case, but with input dimension $d=30$ and a slightly deeper residual network; again, an independent logistic distribution with zero location and unit scale is used as the base measure, and trajectories are integrated from $t=0$ to $t=1$ with the Heun scheme.

Because direct visualization in 30 dimensions is not possible, we assess the generative quality of the learned transport via a principal component analysis (PCA). Figure~\ref{fig:5}(a) shows a scatter plot of the first two principal components computed from reference samples drawn directly from the 30-dimensional mixture, while Fig.~\ref{fig:5}(b) displays the projection of samples generated by the flow-matching model onto the same PCA basis. The two point clouds are nearly indistinguishable, with four clearly separated clusters aligned with the projected component means, indicating that the learned transport map captures the dominant low-dimensional structure of the high-dimensional target distribution and preserves the multimodal geometry under linear projection.

We also study the effect of flow-based transport and RQMC sampling on moment estimation in this higher-dimensional setting. As in the two-dimensional example, we employ the four estimators from Section~\ref{sec:numerics} (FM-MC, FM-QMC, FM-ISMC, and FM-ISQMC) to estimate the first and second moment of the 30-dimensional mixture and compute, for each method and each sample size $N$, the root-mean-square error over $10$ independent repetitions. The resulting log-RMSE curves are reported in Fig.~\ref{fig:6}. The two Monte Carlo-based schemes (FM-MC and FM-ISMC) again display the characteristic $O(N^{-1/2})$ decay. The FM-QMC estimator exhibits a noticeably faster error reduction for moderate sample sizes, reflecting the variance reduction afforded by the scrambled Sobol' design, but its convergence rate eventually levels off as $N$ increases, consistent with the increasing influence of residual approximation error in the learned high-dimensional transport. In contrast, FM-ISQMC maintains a substantially lower RMSE than the Monte Carlo-based methods across the entire sample range and exhibits a markedly steeper decay in the log–log plots, demonstrating that the combination of flow-based importance sampling with RQMC remains effective in improving accuracy even in this 30-dimensional multimodal example.

\begin{figure}[!htbp]
  \centering
  \includegraphics[width=0.75\textwidth]{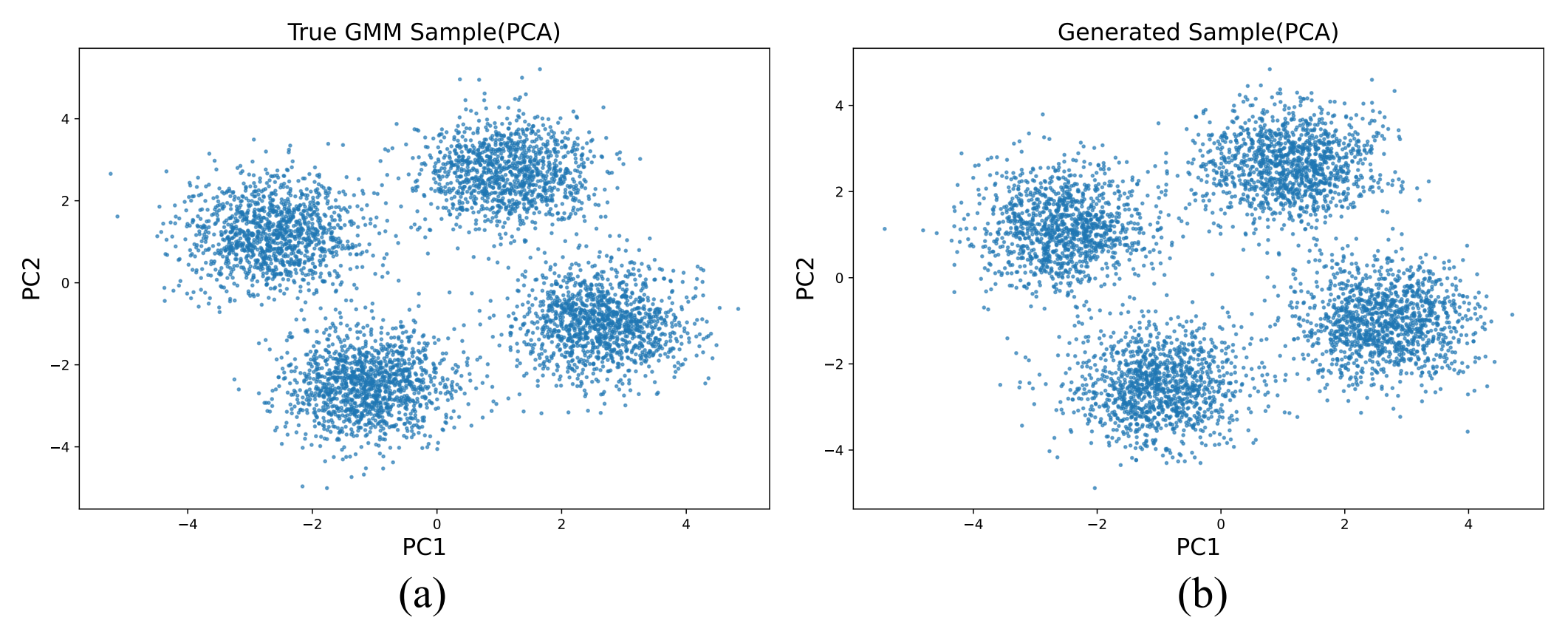}
  \caption{(a) Scatter plot of the first two principal components of samples from the 30-dimensional mixture of Gaussian distribution. (b) Scatter plot of the first two principal components of samples generated by the flow-matching model.}
  \label{fig:5}
\end{figure}

\begin{figure}[!htbp]
  \centering
  \includegraphics[width=0.75\textwidth]{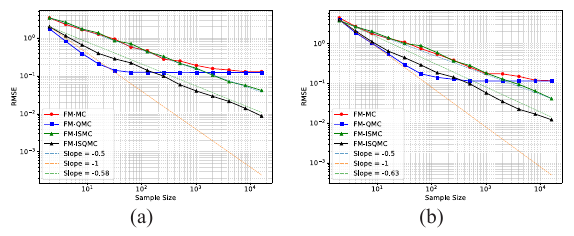}
  \caption{Log-RMSEs for estimating (a) the first and (b) the second moment of the 30-dimensional mixture of Gaussian distribution. Each RMSE is computed over 10 independent repetitions.}
  \label{fig:6}
\end{figure}
\subsection{Banana-shape distribution}
We now consider a two-dimensional banana-shaped distribution, which provides a nonlinear and strongly non-Gaussian benchmark. The target density $\pi$ is defined as the pushforward of a standard bivariate Gaussian under the nonlinear transformation
\[
  z = (z_1,z_2)^\top \sim \mathcal{N}(0, I_2), 
  \qquad
  x_1 = z_1,\qquad
  x_2 = a z_1^2 + c + b z_2,
\]
with parameters $a = 0.3$, $b = 1/\sqrt{2}$, and $c = -1$. Equivalently, the density of $x = (x_1,x_2)^\top$ admits the closed-form expression
\[
  \pi(x_1,x_2)
  \;=\;
  \frac{1}{|b|}
  \,\varphi(x_1)\,
  \varphi\!\left(\frac{x_2 - a x_1^2 - c}{b}\right),
\]
where $\varphi$ denotes the standard normal density on $\mathbb{R}$. This construction yields a curved, banana-shaped distribution that is unimodal in radius but exhibits strong nonlinear dependence between $x_1$ and $x_2$. 

Figure~\ref{fig:3} assesses the quality of the learned transport for this nonlinear target. As shown in Figs.~\ref{fig:3}(a)–(b), the samples generated by the flow-matching model closely reproduce the characteristic banana-shaped support of the true distribution, including its curvature and non-Gaussian tails. Figs.~\ref{fig:3}(c)–(d) compare the analytical banana density with the density estimated from the learned flow. The two contour plots nearly coincide over the region of interest, indicating that the model provides a faithful approximation of $\pi$ both at the level of samples and in terms of the underlying density.

We then use the learned flow to study the impact of flow-based transport and RQMC sampling on the accuracy of moment estimation for this banana-shaped target. The resulting log-RMSE curves in Fig.~\ref{fig:4} shows that the two Monte Carlo-based estimators (FM-MC and FM-ISMC) exhibit the expected $O(N^{-1/2})$ convergence. The FM-QMC estimator initially attains a noticeably higher convergence rate,  reflecting the variance reduction provided by RQMC; however, its error eventually saturates as $N$ increases, due to the residual approximation error in the learned transport map, which induces a non-negligible bias. In contrast, FM-ISQMC achieves a convergence rate that remains significantly better than $O(N^{-1/2})$ over the entire range of sample sizes considered, maintaining a steadily decreasing log-RMSE. This behavior demonstrates that, combining flow-based importance sampling with RQMC can effectively leverage both variance reduction and bias correction to yield highly efficient moment estimates.

\begin{figure}[!htbp]
  \centering
  \includegraphics[width=0.75\textwidth]{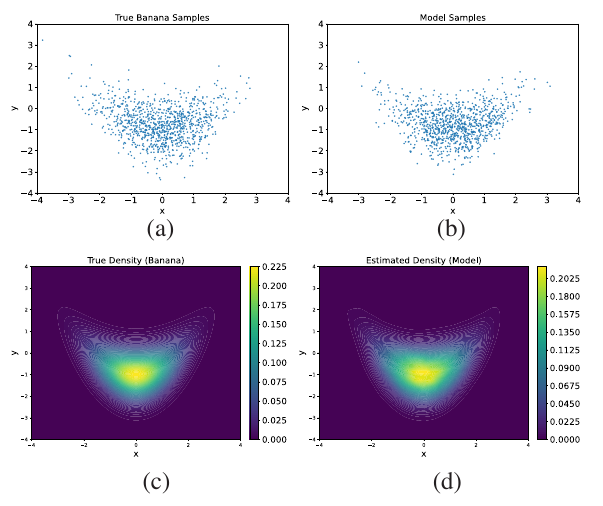}
  \caption{(a) Reference samples from the 2D Banana distribution. (b) Samples generated by the flow-matching model. (c) True density of the 2D Banana distribution. (d) Density estimated by the flow-matching model.}
  \label{fig:3}
\end{figure}

\begin{figure}[!htbp]
  \centering
  \includegraphics[width=0.75\textwidth]{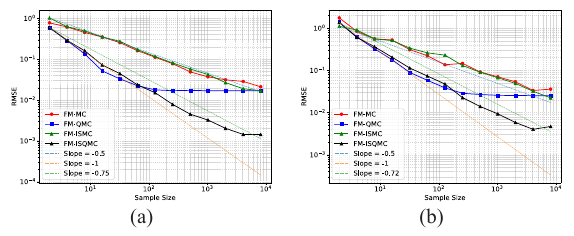}
  \caption{Log-RMSEs for estimating (a) the first and (b) the second moment of the 2D Banana distribution. Each RMSE is computed over 10 independent repetitions.}
  \label{fig:4}
\end{figure}

\section{Conclusions}
\label{sec:conclusions}

In this study, we have investigated how Flow Matching models can be combined with randomized quasi–Monte Carlo (RQMC) methods to construct efficient transport maps for high-dimensional integration. Working in a general importance sampling framework, we analyzed estimators based on a diffeomorphic transport map $\tau$ from the unit cube to a proposal density $q_\tau$ and showed that, under appropriate growth and regularity assumptions on the integrand, the target density, and the derivatives of $\tau$, the resulting RQMC estimators satisfy boundary growth condition. As a consequence, the root-mean-square error (RMSE) of the corresponding RQMC importance sampling estimator achieves the near-optimal convergence rate $O(N^{-1+\varepsilon})$ for any $\varepsilon>0$. We further verified that these assumptions are fulfilled for transport maps obtained from flow matching with a logistic base distribution and an Euler discretization of the probability–flow ODE, provided the learned velocity field has uniformly bounded derivatives of sufficient order.

Our theoretical analysis is complemented by a set of numerical experiments on challenging target distributions, including a 2D Gaussian mixture, a 2D banana-shaped distribution, and a 30D Gaussian mixture. The experiments consistently show that plain Flow Matcing Monte Carlo and RQMC estimators (FM–MC and FM–QMC) are limited by the approximation bias of the learned transport: the Monte Carlo variants exhibit the expected $O(N^{-1/2})$ convergence, while FM–QMC initially benefits from the enhanced variance reduction of RQMC but eventually saturates as $N$ increases. In contrast, our FM–ISQMC, attain convergence rates that are substantially better than $O(N^{-1/2})$ and maintain a steadily decreasing log-RMSE over the entire range of sample sizes considered, thereby demonstrating that combining Flow matching models and  importance sampling-based RQMC can effectively harness both variance reduction and bias correction in practice.

Our theoretical analysis assumes a fixed-step Euler discretization and a uniformly bounded velocity field. In practice, adaptive ODE solvers are often used, and optimal velocity fields may exhibit singularities at the boundaries \cite{gao2024convergence}. Bridging this gap between theory and practice is an important direction for future research. Moreover, a more rigorous tracking of the constants in our derivative growth bounds, particularly their dependence on dimension and number of steps, would strengthen the theoretical foundation. We will consider these aspects in future work.



%% References


\bibliographystyle{elsarticle-num}
\bibliography{qmc}






% \end{thebibliography}

\end{document}